\documentclass[reqno,12pt]{amsart}

\usepackage{amsmath, amsthm, amsfonts, amssymb,color}
\input{mathrsfs.sty}

\usepackage{amsmath}
\usepackage{amsfonts}
\usepackage{amssymb}
\usepackage{eufrak}
\usepackage{amscd}
\usepackage{amsthm}
\usepackage{amstext}
\usepackage[all]{xy}

\newcommand{\supp}{\operatorname{supp}}

   \theoremstyle{plain}%default
   \newtheorem{thm}{Theorem}%[section]
   
   \newtheorem{lem}[thm]{Lemma}
   
   \theoremstyle{definition}

   \theoremstyle{remark}

\author{V. Manuilov}

\date{}

\address{Moscow Center for Fundamental and Applied Mathematics, Moscow State University,
Leninskie Gory 1, Moscow, 
119991, Russia}

\email{manuilov@mech.math.msu.su}

\thanks{Supported by the RSF grant 23-21-00068.}

\title{An example of a continuous field of Roe algebras}

\begin{document}

\begin{abstract}
The Roe algebra $C^*(X)$ is a non-commutative $C^*$-algebra reflecting metric properties of a space $X$, and it is interesting to understand relation between the Roe algebra of $X$ and the (uniform) Roe algebra of its discretization. Here we do a minor step in this direction in the simplest non-trivial example $X=\mathbb R$ by constructing a continuous field of $C^*$-algebras over $[0,1]$ with the fibers over non-zero points the uniform $C^*$-algebra of the integers, and the fiber over 0 a $C^*$-algebra related to $\mathbb R$.

\end{abstract}

\maketitle

\section{Introduction}

Roe algebras play an increasingly important role in the index theory of elliptic operators on noncompact manifolds and their generalizations \cite{Roe,Higson-Roe,Ludewig, Meyer}. Following the ideology of noncommutative geometry \cite{Connes}, they provide an interplay between metric spaces (e.g. manifolds) and (noncommutative) $C^*$-algebras. 

Let $X$ be a proper metric measure space, that is, $X$ is a set, which is equipped with a metric $d$ and a measure $m$ defined on the Borel $\sigma$-algebra defined by the topology on $X$ induced by the metric, and all balls are compact.  
For a Hilbert space $H$ we write $\mathbb B(H)$ (resp., $\mathbb K(H)$) for the algebra of all bounded (resp., all compact) operators on $H$.

Recall the definition of the Roe algebra of $X$ \cite{Roe}.
Let $H_X$ be a Hilbert space with an action of the algebra $C_0(X)$ of continuous functions on $X$ vanishing at infinity (i.e. a $*$-homomorphism $\pi:C_0(X)\to\mathbb B(H_X)$). We will assume that 
\begin{equation}\label{0}
\{\pi(f)\xi:f\in C_0(X),\xi\in H_X\} \quad\mbox{is\ dense\ in}\quad H_X 
\end{equation}
and that 
\begin{equation}\label{1}
\pi(f)\in\mathbb K(H_X)\quad \mbox{implies\ that}\quad f=0. 
\end{equation}
An operator $T\in\mathbb B(H_X)$ is {\it locally compact} if the operators $T\pi(f)$ and $\pi(f)T$ are compact for any $f\in C_0(X)$. It has {\it finite propagation} if there exists some $R>0$ such that $\pi(f)T\pi(g)=0$ whenever the distance between the supports of $f,g\in C_0(X)$ is greater than $R$. The {\it Roe algebra} $C^*(X,H_X)$ is the norm completion of the $*$-algebra of locally compact, finite propagation operators on $H_X$. As it does not depend on the choice of $H_X$ satisfying (\ref{0}) and (\ref{1}), it is usually denoted by $C^*(X)$. If $X=\mathbb R$ with the standard metric and the standard measure (our main example) then we may (and will, for simplicity) take $H_X=L^2(X)$. 

When $X$ is discrete, the choice $H_X=l^2(X)$ does not satisfy the condition (\ref{1}). In order fix this, one may take $H_X=l^2(X)\otimes H$ for an infinitedimensional Hilbert space $H$. But there is also another option: still to use $H_X=l^2(X)$. The resulting $C^*$-algebra is called the {\it uniform Roe algebra} of $X$, and is denoted by $C^*_u(X)$. This $C^*$-algebra is more tractable, but has less relations with elliptic theory. 

Manifolds and some other spaces $X$ are often endowed with discrete subspaces $D\subset X$ that are $\varepsilon$-dense for some $\varepsilon$, e.g. $\mathbb Z\subset\mathbb R$, or, more generally, lattices in Lie groups, or, even more generally, Delone sets in metric spaces \cite{Bellisard}. Some problems related to $X$ may become simpler when reduced to $D$ (discretization). In particular, it would be interesting to understand relation between a Roe-type algebra of $X$ and the uniform Roe algebra of its discretization $D$. As the first step, we consider here one of the simplest non-trivial cases, $X=\mathbb R$, $D=\mathbb Z$, and construct a continuous field of $C^*$-algebras over the segment $[0,1]$ such that the fiber over $0$ is a certain $C^*$-algebra related to $\mathbb R$, while the fiber over any other point is the uniform Roe algebra of $\mathbb Z$. Such non locally trivial continuous fields of $C^*$-algebras are interesting because they provide relations between fibers over different points. In particular, they provide a map from the $K$-theory group of the fiber over $0$ to the $K$-theory group of the fiber over non-zero points. A similar continuous field with the fiber over $0$ the algebra of functions on a sphere and the fibers over non-zero points the algebra of compact operators was used in \cite{Nest} to give a proof of Bott periodicity in $K$-theory.

\section{Two maps}

Let $D_t=t\mathbb Z\subset\mathbb R$. In this section we construct the maps $\alpha_t:C^*_u(D_t)\to C^*(\mathbb R)$ and $\beta_t: C^*(\mathbb R)\to C^*_u(D_t)$, $t\in(0,1]$.

Let $\varphi_0(x)=\left\lbrace\begin{array}{cl}1+x,& x\in[-1,0];\\1-x,& x\in[0,1];\\0,& \mbox{otherwise},\end{array}\right.$ $\varphi_n(x)=\varphi_0(x-n)$, $\varphi_n^t(x)=\frac{1}{\sqrt{t}}\varphi_n(x/t)$. Then $\supp\varphi_n^t=[t(n-1),t(n+1)]$,  and $\|\varphi^t_n\|_{L^2}=\sqrt{2/3}$ for any $n\in\mathbb Z$ and any $t\in(0,1]$ (here $\|\cdot\|_{L^2}$ denotes the norm in $L^2(\mathbb R)$). In particular, $\varphi_n^t$ and $\varphi_m^t$ are orthogonal when $|m-n|\geq 2$. Let $p_t$ denote the projection, in $L^2(\mathbb R)$, onto the closure $H_t$ of the linear span of $\varphi_n^t$, $n\in\mathbb Z$. %For $t=0$ we set $p_0=1$.

Let $(G_{nm})_{n,m\in\mathbb Z}$ be the Gram matrix for $\{\varphi_n^t\}$, $n\in\mathbb Z$, $G_{nm}=\langle\varphi_n,\varphi_m\rangle$, (note that $G$ does not depend on $t$) and let $G\in\mathbb B(l^2(\mathbb Z))$ be the operator with the Gram matrix with respect to the standard basis of $l^2(\mathbb Z)$. 
%We consider $G^t$ as an operator on $H_t$.

\begin{lem}
The operator $G$ is bounded, invertible and has finite propagation.

\end{lem} 
\begin{proof}
Direct calculation shows that $G_{n,n}=\frac{2}{3}$, $G_{n,n\pm 1}=\frac{1}{6}$, and $G_{n,m}=0$ when $|m-n|\geq 2$. Therefore, $\|G\|\leq \frac{2}{3}+\frac{1}{3}$ and $\|\frac{2}{3}-G\|=\frac{1}{3}<1$, hence $G$ is invertible.
\end{proof}
 
Set $C=G^{-1/2}$. By functional calculus, $C$ can be approximated by polynomials in $G$, hence $C$ lies in the norm closure of operators of finite propagation, i.e. $C\in C^*_u(\mathbb Z)$.

Let $A\in\mathbb B(l^2(\mathbb Z))$, and let $A_{nm}$ be its matrix elements with respect to the standard basis. Define $\gamma_t(A)\in\mathbb B(H_t)$ by $\gamma_t(A)\varphi_m^t=\sum_{n,m\in\mathbb Z}A_{nm}\varphi_n^t$. Note that $\gamma_t$ is a homomorphism, but not a $*$-homomorphism.

\begin{lem}
There exist $k_1,k_2>0$ such that $k_1\|A\|<\|\gamma_t(A)\|<k_2\|A\|$.

\end{lem}
\begin{proof}
Let $S$ denote the right shift on $l^2(\mathbb Z)$, $x=\sum_{i\in\mathbb Z}x_i\varphi_i^t$. Then 
\begin{eqnarray*}
\|\gamma_t(A)x\|^2&=&\sum_{i,j,k,l\in\mathbb Z}\bar x_i x_j\bar A_{ki}A_{lj}\langle\varphi_k^t,\varphi_l^t\rangle\\
&=&\sum_{i,j,k\in\mathbb Z}\bar x_i x_j\bar A_{ki}A_{kj}+\frac{1}{6}\sum_{i,j,k,l\in\mathbb Z}\bar x_i x_j\bar A_{ki}A_{k\pm 1,j}\\
&=&\frac{2}{3}\|A\tilde x\|+\frac{1}{6}\langle A\tilde x,(S+S^*)A\tilde x\rangle,
\end{eqnarray*}
where $\tilde x\in l^2(\mathbb Z)$ has coordinates $x_i$ with respect to the standard basis of $l^2(\mathbb Z)$.
As $|\frac{1}{6}\langle A\tilde x,(S+S^*A\tilde x\rangle|\leq \frac{1}{3}\|A\|\|\tilde x\|$, the conclusion follows.
\end{proof}

Set $\psi_n^t=\gamma_t(C)\varphi_n^t=\sum_{m\in\mathbb Z}C_{mn}\varphi_m^t$. Then $\langle\psi_n^t,\psi_m^t\rangle=\langle \gamma_t(G^{-1})\varphi_n,\varphi_m\rangle=\delta_{n,m}$, hence $\{\psi_n^t\}_{n\in\mathbb Z}$ is an orthonormal system. Invertibility of $C$ implies that the closures of the linear spans of $\{\varphi_n^t\}_{n\in\mathbb Z}$ and of $\{\psi_n^t\}_{n\in\mathbb Z}$ coincide. The advantage of this orthonormal system with respect to the system obtained from $\{\varphi_n^t\}_{n\in\mathbb Z}$ by Gram--Schmidt orthogonalization is that it is obtained from the original non-orthogonal system by an operator from $C^*_u(\mathbb Z)$.

Define a map $\alpha_t:C^*_u(\mathbb Z)\to C^*(\mathbb R)$. Let $T\in C^*_u(\mathbb Z)$, $T=(T_{nm})_{n,m\in\mathbb Z}$. Set 
$$
\alpha_t(T)(f)=\sum_{n,m\in\mathbb Z}T_{nm}\psi_n^t\langle \psi_m^t,f\rangle, \quad f\in L^2(\mathbb R).
$$ 
Let $U_t:l^2(\mathbb Z)\to L^2(\mathbb R)$ be the isometry defined by $U_t(\delta_n)=\psi_n^t$. Then it is easy to see that $\alpha_t(T)=U_tTU_t^*$. Hence $\alpha_t$ is a $*$-homomorphism, in particular, it is isometric.
As $T$ is bounded, $\alpha_t(T)$ is bounded as well. 

As $\gamma_t(C)$ can be considered as the transition matrix from the basis $\{\varphi_n^t\}$ to the basis $\{\psi_n^t\}$, we can write $\alpha_t(T)=\gamma_t(C)^{-1}\gamma_t(T)\gamma_t(C)$.

%\color{red}{$\gamma$ is a homomorphism !!!}\color{black}

It remains to check that $\alpha_t(T)\in C^*(\mathbb R)\subset\mathbb B(L^2(\mathbb R))$. To this end, consider one more basis for $H_t$. By construction of $C$, for any $\varepsilon>0$ there exists an operator $C_\varepsilon\in\mathbb B(l^2(\mathbb Z))$ of finite propagation $M_\varepsilon$ such that $\|C-C_\varepsilon\|<\varepsilon$.  Set $\psi_n^{t,\varepsilon}=\gamma_t(C_\varepsilon)\varphi_n^t$. Set $\tilde T_\varepsilon=\gamma_t(C_\varepsilon)^{-1}\gamma(T)\gamma(C_\varepsilon)$. 

\begin{lem}
For sufficiently small $\varepsilon$ there exists $K>0$ such that $\|\alpha_t(T)-\tilde T_\varepsilon\|<K\varepsilon$ for any $t\in(0,1]$.

\end{lem}
\begin{proof}
One should take $\varepsilon$ small enough to provide invertibility of $\gamma(C_\varepsilon)$. 
Then
\begin{eqnarray*}
\|\alpha_t(T)-\tilde T_\varepsilon\|&\leq &\|\gamma_t(C)^{-1}-\gamma_t(C_\varepsilon)^{-1}\|\cdot\|\gamma_t(T)\|\cdot\|\gamma_t(C)\|
\\
&&+
\|\gamma_t(C_\varepsilon)^{-1}\|\cdot\|\gamma_t(T)\|\cdot\|\gamma_t(C)-\gamma_t(C_\varepsilon)\|\\
&\leq&\|\gamma_t(C-C_\varepsilon)\|\cdot\|\gamma_t(C)^{-1}\|\cdot\|\gamma_t(C_\varepsilon)^{-1}\|\cdot\|\gamma_t(C)\|
\\
&&+\|\gamma_t(C_\varepsilon)^{-1}\|\cdot\|\gamma_t(T)\|\cdot\|\gamma_t(C-C_\varepsilon)\|\\
&<&k_2\varepsilon(\|\gamma_t(C)^{-1}\|\cdot\|\gamma_t(C_\varepsilon)^{-1}\|\cdot\|\gamma_t(C)\|+\|\gamma_t(C_\varepsilon)^{-1}\|\cdot\|\gamma_t(T)\|).
\end{eqnarray*}
\end{proof}

\begin{lem}
$\alpha_t(T)\in C^*(\mathbb R)$ for any $T\in C^*_u(\mathbb Z)$.

\end{lem}
\begin{proof}
As $T\in C^*_u(\mathbb Z)$, it can be approximated by finite propagation operators $T^N$, $N\in\mathbb N$, with propagation $N$. This means that the matrix of $T^N$ has the band structure ($T^N_{nm}=0$ when $|m-n|>c$ for some $c>0$). Then we may write $T^N$ as a matrix with $2N+1$ diagonals: $T^N\delta_n=\sum_{k=-N}^N\lambda_{n,k}\delta_{n+k}$, where the numbers $\lambda_{n,k}$ are uniformly bounded by $\|T\|$. 

As $\alpha_t(T^N)$ can be approximated by operators of the form $\tilde T^N_\varepsilon$, it suffices to show that $\tilde T^N_\varepsilon\in C^*(\mathbb R)$.

Let $f\in C_0(\mathbb R)$ has compact support, say $[a,b]\subset\mathbb R$. Then 
\begin{eqnarray}\label{4}
\tilde T^N_\varepsilon\pi(f)(g)&=&\sum_{n,m\in\mathbb Z}T^N_{nm}\langle \psi_m^{t,\varepsilon},fg\rangle\psi_n^{t,\varepsilon}=\sum_{n,m\in\mathbb Z}T^N_{nm}\langle \gamma_t(C\varepsilon)\varphi_m^t,fg\rangle \gamma_t(C_\varepsilon)\varphi_n^t\nonumber
\\
&=&\sum_{n\in\mathbb Z}\sum_{k=-N}^N\lambda_{n,k}\langle \gamma_t(C_\varepsilon)\varphi_{n+k}^t,fg\rangle \gamma_t(C_\varepsilon)\varphi_n^t.
\end{eqnarray}    

As $\supp(fg)\subset[a,b]$ and as propagation of $C_\varepsilon\leq M_\varepsilon$, we have $\supp(\gamma_t(C_\varepsilon)\varphi_n^t)\subset[t(n-1-M_\varepsilon),t(n+1+M_\varepsilon)]$.  Therefore, $\langle \gamma_t(C_\varepsilon)\varphi_{n+k}^t,fg\rangle\neq 0$ only when $[a,b]\cap [t(n-1-M_\varepsilon),t(n+1+M_\varepsilon)]\neq\emptyset$, thus the sum (\ref{4}) contains only finite number of non-zero summands, i.e. $\operatorname{Ran}\tilde T^N_\varepsilon\pi(f)$ is finitedimensional. Similarly, $\operatorname{Ran}\pi(f)\tilde T^N_\varepsilon$ is finitedimensional. Thus $\alpha_t(T)\pi(f)$ and $\pi(f)\alpha_t(T)$ are compact. Approximation of functions in $C_0(\mathbb R)$ by functions $f$ with finite support proves that $\alpha_t(T)$ is locally compact.

Similarly one can show that $\alpha_t(T)$ is of finite propagation. Indeed, let $f,g\in C_0(\mathbb R)$ are such that the distance between their supports is greater than $R$. Then
$$  
\pi(f)\alpha_t(T^N_\varepsilon)\pi(g)(h)=\sum_{n\in\mathbb Z}\sum_{k=-N}^N\lambda_{n,k}\langle \gamma_t(C_\varepsilon)\varphi_{n+k}^t,gh\rangle f\gamma_t(C_\varepsilon)\varphi_n^t.
$$
We have $\gamma_t(C_\varepsilon)\varphi_{n+k}^t,gh\rangle=0$ when $\supp g\cap[t(n-k-1-N),t(n+k+1+N)]=\emptyset$, while $f\gamma_t(C_\varepsilon)\varphi_n^t=0$ when $\supp f\cap[t(n-1-M_\varepsilon),t(n+1+M_\varepsilon)]=\emptyset$, so if $R$ is sufficiently great then their product vanishes.
\end{proof} 

The second map, $\beta_t:C^*(\mathbb R)\to C^*_u(\mathbb Z)$, goes in the opposite direction and is not a homomorphism (but linear and even completely positive). In fact, it extends to a completely positive map from a greater $C^*$-algebra $C^*_p(\mathbb R)\supset C^*(\mathbb R)$, which is the norm closure of all bounded operators of finite propagation without the requirement of local compactness . 
For $S\in C^*_p(\mathbb R)$ set $(\beta_t(S))_{nm}=\langle \psi_n^t,S\psi_m^t\rangle$. Then the operator $\beta_t(S)$ can be written as $\beta_t(S)(\delta_m)=\sum_{n\in\mathbb Z}\langle \psi_n^t,S\psi_m^t\rangle\delta_n$. Recall that we denote by $U_t:l^2(\mathbb Z)\to L^2(\mathbb R)$ the isometry that maps the standard basis $\{\delta_n\}_{n\in\mathbb Z}$ of $l^2(\mathbb Z)$ to the basis $\{\psi_n^t\}_{n\in\mathbb Z}$ of $H_t\subset L^2(\mathbb R)$. Then $\beta_t(S)=U^*_tSU$. In particular, this implies that $\beta_t(S)$ is bounded for any bounded operator $S$.

\begin{lem}\label{L5}
Let $S\in C^*_p(\mathbb R)$. Then $\beta_t(S)\in C^*_u(\mathbb Z)$ for any $t>0$.

\end{lem}
\begin{proof}
It suffices to show that $\beta_t(S)\in C^*_u(\mathbb Z)$ for operators of finite propagation. For an operator $S$ of finite propagation set $\tilde S=U_t^*D^*SDU_t$, where $D=C_\varepsilon C^{-1}$. As $\|1-D\|<\varepsilon\|C^{-1}\|$, $\beta_t(S)$ can be approximated by operators of the form $\tilde S$. Let us show that $\tilde S$ has finite propagation, which means, for discrete spaces, that the matrix of this operator is a band matrix. We have 
$$
\tilde S_{nm}=\langle \psi_n^t,D^*SD\psi_m^t\rangle=\langle D\psi_n^t,SD\psi_m^t\rangle=\langle \psi_n^{t,\varepsilon},S\psi_m^{t,\varepsilon}\rangle.
$$
As $\supp\psi_n^{t,\varepsilon}\in[t(n-1-M_\varepsilon),t(n+1+M_\varepsilon)]$ and as $S$ has finite propagation, $\tilde S_{nm}=0$ when $|n-m|$ is sufficiently great.
\end{proof}

Note that $\beta_t\circ\alpha_t(S)=p_tS|_{H_t}$, in particular, this means that $p_tS|_{H_t}$ is locally compact for any $S\in C^*_p(\mathbb R)$.

\section{The fiber over 0}

Let $L_0^\infty(\mathbb R)$ denote the norm closure of $\cup_{N}L^\infty([-N,N])\subset L^\infty(\mathbb R)$. The group $\mathbb R$ acts on $L_0^\infty(\mathbb R)$ by translations. Set $A_0=L^\infty_0(\mathbb R)\rtimes\mathbb R$. 

\begin{lem}
$A_0\subset C^*_p(\mathbb R)$.

\end{lem}
\begin{proof}
Let $f\in L^\infty([-N,N])$, and let $g\in C_0(\mathbb R)$ be a continuous function with compact support. The linear combinations of operators of the form $S_{f,g}$, where 
$$
S_{f,g}(u)(x)=\int f(x)g(y)u(x-y)\,dy, 
$$
are dense in $L^\infty_0(\mathbb R)\rtimes\mathbb R$, so it suffices to show that $S_{f,g}\in C^*_p(\mathbb R)$. Let $\operatorname{supp}(g)\subset [-M,M]$, and let $\varphi,\psi\in C_0(\mathbb R)$ have supports at the distance greater than $L$. Then 
$$
(\pi(\varphi) S_{f,g}\pi(\psi) (u))(x)=\varphi(x)\int f(x)g(y)\psi(x-y)u(x-y)\,dy=0
$$ 
if $L>M$.   
\end{proof} 

Recall that $C$ is the transition matrix that maps $\varphi_n^t$ to $\psi_n^t$, i.e. $\psi_n^t=\sum_{m\in\mathbb Z}C_{mn}\varphi_m^t$. We have defined $C$ by $C=G^{-1/2}$, where $G$ is the Gram matrix for $\{\varphi_n^t\}_{n\in\mathbb N}$. We need the following technical result. 

\begin{lem}\label{lemmaC}
The series $\sum_{n\in\mathbb Z}|C_{nm}|$ and $\sum_{m\in\mathbb Z}|C_{nm}|$ converge. The sums $\sum_{n\in\mathbb Z}|C_{nm}|$ (resp., $\sum_{m\in\mathbb Z}|C_{nm}|$) are bounded uniformly with respect to $m$ (resp., to $n$).

\end{lem}
\begin{proof}
When working with matrices with the same entries along any diagonal it is convenient to identify $l^2(\mathbb Z)$ with the square-integrable functions on the circle, and the basis $\{\delta_n\}_{n\in\mathbb Z}$ with the basis $\{e^{inx}\}$. Under this identification, the matrix $B_{nm}=b_{n-m}$ can be identified with the operator of multiplication by the function $\sum_{n\in\mathbb N}b_ne^{inx}$. Thus, the Gram matrix $G$ corresponds to the invertible function $\frac{2}{3}+\frac{1}{3}\cos x$, and the matrix $C$ corresponds to the function $(\frac{2}{3}+\frac{1}{3}\cos x)^{-1/2}$. As this function is smooth, its Fourier coefficients $a_n$, $n=0,1,\ldots$, are of rapid decay, i.e. $a_n=o(n^{-k})$ for any $k\in\mathbb N$. Therefore, the series $\sum_{n\in\mathbb N}|a_n|$ is convergent. As $C_{nm}=a_{|n-m|}$, the series $\sum_{n\in\mathbb Z}|C_{nm}|$ and $\sum_{m\in\mathbb Z}|C_{nm}|$ converge. Uniform boundedness is obvious.
\end{proof}

Denote the map $t\mapsto\beta_t(S)$ by $\beta^S:(0,\infty)\to C^*_u(\mathbb Z)$.

\begin{thm}\label{continuity}
The map $\beta^S$ is norm-continuous on $(0,\infty)$ for any $S\in A_0$. 

\end{thm}
\begin{proof}
Note that the linear combinations of operators $S_{f,g}$, $S_{f,g}(u)(x)=\int f(x)g(y)u(x-y)\,dy$, with $f\in L^\infty(\mathbb R)$, $g\in C_0(\mathbb R)$ of finite support are dense in $A_0=L^\infty_0(\mathbb R)\rtimes\mathbb R$, so it suffices to show continuity of the map $t\mapsto\beta_t(S)$ for $S=S_{f,g}$ for $f$ and $g$ with compact support.

Let $\|S_{f,g}\|=1$, $\operatorname{supp}(f),\operatorname{supp}(g)\subset[-N,N]$, $a=\sum_{n\in\mathbb Z}a_n\delta_n\in l^2(\mathbb Z)$, $\|a\|=1$. Then
\begin{eqnarray}
\|(\beta_t(S_{f,g})-\beta_{t_0}(S_{f,g}))a\|^2&=&\sum_{n\in\mathbb Z}\Bigl(\sum_{m\in\mathbb Z}(\langle\psi_n^t,S_{f,g}\psi_m^t\rangle-\langle\psi_n^{t_0},S_{f,g}\psi_m^{t_0}\rangle)a_m\Bigr)^2\nonumber\\
&\leq&\sum_{n\in\mathbb Z}\Bigl(\sum_{m\in\mathbb Z}\langle\psi_n^t-\psi_n^{t_0},S_{f,g}\psi_m^t\rangle a_m\Bigr)^2\label{111}\\
&+&\sum_{n\in\mathbb Z}\Bigl(\sum_{m\in\mathbb Z}(\langle\psi_n^{t_0},S_{f,g}(\psi_m^t-\psi_m^{t_0})\rangle a_m\Bigr)^2.\label{112}
\end{eqnarray}
We shall estimate the first summand (\ref{111}). The second summand (\ref{112}) can be estimated in the same way (or, passing to the adjoint of $S_{f,g}$).

Recall that $\psi_n^t=\sum_{k\in\mathbb Z}C_{kn}\varphi_k^t$. Then
\begin{equation}\label{113}
\sum_{n\in\mathbb Z}\Bigl(\sum_{m\in\mathbb Z}\langle\psi_n^t-\psi_n^{t_0},S_{f,g}\psi_m^t\rangle a_m\Bigr)^2=
\sum_{n\in\mathbb Z}\Bigl(\sum_{m,k,l\in\mathbb Z}C_{kn}C_{lm}\langle\varphi_k^t-\varphi_k^{t_0},S_{f,g}\varphi_l^t\rangle a_m\Bigr)^2.
\end{equation}

Let $t\in[\frac{t_0}{2},2t_0]$.
As the supports of $f$ and $g$ lie in $[-N,N]$, $S_{f,g}\varphi_l^t=0$ for $|l|>(N+2)/t$, hence the sum over $l$ is finite, over $|l|\leq 2(N+2)/t_0$. Also the support of $S_{f,g}\varphi_l^t$ lies in $[(l-1)t-N,(l+1)t+N)]$, hence there are only finitely many $k$ such that $\langle\varphi_k^t-\varphi_k^{t_0},S_{f,g}\varphi_l^t\rangle\neq 0$. In other words, the sum in
(\ref{113}) can be written as 
\begin{equation}\label{114}
\sum_{n\in\mathbb Z}\Bigl(\sum_{|k|,|l|\leq M}\sum_{m\in\mathbb Z}C_{kn}C_{lm}\langle\varphi_k^t-\varphi_k^{t_0},S_{f,g}\varphi_l^t\rangle a_m\Bigr)^2
\end{equation}
for some $M$.

For any $\varepsilon>0$ there exists $\delta>0$ such that $\|\varphi_k^t-\varphi_k^{t_0}\|_{L^2}<\frac{\varepsilon}{M^2}$ for any $|k|\leq M$ when $|t-t_0|<\delta$.

Fix $k$ and $l$, and estimate 
\begin{eqnarray*}
&&\sum_{n\in\mathbb Z}\Bigl(\sum_{m\in\mathbb Z}C_{kn}C_{lm}\langle\varphi_k^t-\varphi_k^{t_0},S_{f,g}\varphi_l^t\rangle a_m\Bigr)^2\\
&&\leq
\sum_{n\in\mathbb Z}\Bigl(\sum_{m\in\mathbb Z}|C_{kn}C_{lm}|\frac{\varepsilon}{M^2}\|S_{f,g}\|\|\varphi_l^t\|_{L^2}\|a\|\Bigr)^2=
\sum_{n\in\mathbb Z}|C_{kn}|\Bigl(\sum_{m\in\mathbb Z}|C_{lm}|\Bigr)^2\frac{\varepsilon^2}{M^4}.
\end{eqnarray*}

By Lemma \ref{lemmaC} the series $\sum_{m\in\mathbb Z}|C_{lm}|$ converges, hence is bounded by some $L$, hence 
$$
\sum_{n\in\mathbb Z}|C_{kn}|\Bigl(\sum_{m\in\mathbb Z}|C_{lm}|\Bigr)^2\frac{\varepsilon^2}{M^4}\leq L^3\frac{\varepsilon^2}{M^4}. 
$$ 
For shortness' sake set $x_{nmkl}=C_{kn}C_{lm}\langle\varphi_k^t-\varphi_k^{t_0},S_{f,g}\varphi_l^t\rangle a_m$. We have shown that 
$$
\sum_{n\in\mathbb Z}\Bigl(\sum_{m\in\mathbb Z}x_{nmkl}\Bigr)^2\leq L^3\varepsilon^2/M^2
$$
for any $k,l$.

Coming back to (\ref{114}) and using $2xx'\leq x^2+(x')^2$, we have 
\begin{eqnarray*}
\sum_{n\in\mathbb Z}\Bigl(\sum_{|k|,|l|\leq M}\sum_{m\in\mathbb Z}x_{nmkl}\Bigr)^2&=&\sum_{n,m,m',k,k',l,l'}x_{nmkl}x_{nm'k'l'}\\
&\leq& M^2\sum_{n,m,m',k,l}x_{nmkl}x_{nm'kl}\\
&=&M^2\sum_{k,l}\sum_n\Bigl(\sum_m x_{nmkl}\Bigr)^2\\
&\leq& M^4\cdot L^3\varepsilon^2/M^4=L^3\varepsilon^2.
\end{eqnarray*}

Thus, for $|t-t_0|<\delta$ we have $\|\beta_t(S_{f,g})-\beta_{t_0}(S_{f,g})\|^2<2L^3\varepsilon^2$ which proves continuity.
\end{proof}

\section{Continuous field of Roe algebras}

Continuous fields of $C^*$-algebras (aka bundles of $C^*$-algebras or $C(T)$-$C^*$-algebras) were introduced by Fell \cite{Fell} and Dixmier (\cite{Dixmier}, Section 10). Recall that a continuous field of $C^*$-algebras over a locally compact Hausdorff space $T$ is a triple $(T,A,\pi_t:A\to A_t)$, where $A$ and $A_t$, $t\in A$, are $C^*$-algebras, the $*$-homomorphisms $\pi_t$ are surjective, the family $\{\pi_t\}_{t\in T}$ is faithful, and the map $t\mapsto\|\pi_t(a)\|$ is continuous for any $a\in A$. 

Set $T=[0,1]$, $A_t=C^*_u(\mathbb Z)$ for $t\neq 0$. The fiber $A_0$ over $0$ was defined in the previous section. 
Set 
$$
A=C_0((0,1];C^*_u(\mathbb Z))+\{\beta^S:S\in A_0\}\subset \prod_{t\in T}A_t.
$$

\begin{lem}
The set $A$ is norm closed.

\end{lem}
\begin{proof}
First, let us show that 
\begin{equation}\label{121}
\sup_{t\in(0,1]}\|\beta_t(S)\|=\|S\|=\lim_{t\to 0}\|\beta_t(S)\|. 
\end{equation}
Consider the projections $p_t=U_tU_t^*$ in $L^2(\mathbb R)$. Note that $\|U_t^*SU_t\|=\|p_tSp_t\|$. Let $f\in C_0(\mathbb R)$ be a Lipschitz function with finite support $[a,b]$, and let $L$ be the Lipschitz constant for $f$. Let $g_t=\sum_{n\in\mathbb Z}\sqrt{t}f(tn)\varphi_n^t$ be a piecewise linear function such that $g_t(tn)=f(tn)$ for any $n\in\mathbb N$. Then $|f(x)-g_t(x)|_{L^2}\leq Lt$ for any $x\in\mathbb R$, hence $\|f-g_t\|_{L^2}\leq Lt\sqrt{b-a+2}$. As $g_t$ lies in the linear span of the functions $\varphi_n^t$, $n\in\mathbb N$, we have $g_t=p_tg_t$. As $\|f-p_tf\|_{L^2}\leq \|f-p_tg_t\|_{L^2}$, we have $\lim_{t\to 0}f-p_tf=0$. As Lipschitz functions with finite support are dense in $L^2(\mathbb R)$, we conclude that the $*$-strong limit of $p_t$ is the identity operator. Note also that $\varphi_n^{2t}=\sqrt{2}\varphi_{2n}^t+\frac{\sqrt{2}}{2}\varphi_{2n-1}^t+\frac{\sqrt{2}}{2}\varphi_{2n+1}^t$, hence the linear span of $\{\varphi_n^t\}$ lies in the linear span of $\{\varphi_n^{t/2}\}$, therefore $p_t\leq p_{t/2}$ for any $t\in(0,1]$, hence the sequence $\|p_{t/2^k}Sp_{t/2^k}\|$ is increasing. 

Consider the norm closure $\overline A$ of $A$. Then $I=C_0((0,1];C^*_u(\mathbb Z))$ is a closed ideal in $\overline A$. Let $\{f_n+\beta^{S_n}\}$ be a Cauchy sequence in $A$. Passing to the quotient $C^*$-algebra $\overline A/I$, the sequence $\{f_n+\beta^{S_n}+I\}=\{\beta^{S_n}+I\}$ is also a Cauchy sequence, as the quotient $*$-homomorphisms have norm 1 (\cite{Dixmier}, Section 1.8). Note that
\begin{eqnarray*}
\|\beta^{S}+I\|=\inf_{f\in I}\|f+\beta^{S}\|\geq\lim_{t\to 0}\|f(t)+\beta_t(S)\|=\lim_{t\to 0}\|\beta_t(S)\|=\|S\|,
\end{eqnarray*} 
hence $\{S_n\}$ is also a Cauchy sequence. As $A_0$ is norm closed, it has a limit in $A_0$. But then $\{f_n\}$ is also a Cauchy sequence, and as $I$ is norm closed, its limit lies in $I$. Thus $\{f_n+\beta^{S_n}\}$ converges in $A$.
\end{proof}

Define $\pi_t:A\to A_t$ by $\pi_t(f+\beta^S)=f(t)+\beta_t(S)$ for $t>0$, and $\pi_0(f+\beta^S)=S$. These maps are well defined as $f_1+\beta^{S_1}=f_2+\beta^{S_2}$ implies that $f_1=f_2$ and $S_1=S_2$.

\begin{thm}
The triple $(T,A,\pi_t:A\to A_t)$ is a continuous field of $C^*$-algebras.

\end{thm}
\begin{proof}
Each $\pi_t$ is clearly surjective. If $\pi_t(f_1+\beta^{S_1})=\pi_t(f_2+\beta^{S_2})$ for any $t\in T$ then, taking $t=0$, we conclude that $S_1=S_2$. Then we see that $f_1(t)=f_2(t)$ for any $t\in(0,1]$, hence $f_1=f_2$. Finally, we have to check that the map $t\mapsto \pi_t(a)$ is continuous. Let $a=f+\beta^S$. Continuity at $t>0$ follows from continuity of $f$ (by definition) and continuity of $\beta^S$ (Theorem \ref{continuity}). Continuity at 0 follows from (\ref{121}).  
\end{proof}

%%%%%%%%%%%%%%%%%%%%%%%%%%%%%%%%%%%%%%%%%%


\begin{thebibliography}{9}

\bibitem{Bellisard}
J. V. Bellisard. {\it Delone Sets and Material Science: a Program.} in Mathematics of Aperiodic Order. Progress in Mathematics, {\bf 309}, 405--428, Birkh\"auser, 2015.

\bibitem{Connes}
A. Connes. {\it Noncommutative Geometry.} Academic Press,  1995.

\bibitem{Dixmier}
J. Dixmier. {\it $C^*$-Algebras}. North-Holland,  1977.

\bibitem{Higson-Roe}
 N. Higson, J. Roe. {\it Analytic K-Homology}. Oxford Mathematical Monographs, OUP Oxford,  2000.

\bibitem{Ludewig}
M. Ludewig, G. C. Thiang. {\it Large-scale geometry obstructs localization.} 
 J. Math. Physics, {\bf 63} (2022), 091902.

\bibitem{Meyer}
E. E. Ewert, R. Meyer. {\it Coarse geometry and topological phases.} Commun.  Math. Physics, {\bf 366} (2019), 1069--1098.

\bibitem{Nest}
G. A. Elliott, T. Natsume, R. Nest. {\it The Heisenberg group and $K$-theory.} $K$-Theory, {\bf 7 } (1991), 409--428.

\bibitem{Fell}
J. M. G. Fell. {\it The structure of algebras of operator fields.} Acta Math. {\bf 106} (1961), 233--280.

\bibitem{Roe}
 J. Roe. {\it Index theory, coarse geometry, and topology of manifolds.} CBMS Regional Conference Series in Mathematics, {\bf 90}. Amer. Math. Soc., Providence, RI, 1996.  



\end{thebibliography}
\end{document}